\documentclass[12pt]{article}
\usepackage{amsmath,amsthm,mathrsfs}
\usepackage{amsfonts}
\usepackage{latexsym}
\usepackage{cite}
\usepackage{amssymb}
\usepackage{t1enc}
\usepackage[latin1]{inputenc}
\usepackage[english]{babel}
\usepackage[dvips]{graphicx}
\usepackage{graphicx}
\usepackage[natural]{xcolor}
\usepackage{tikz,color}
\DeclareMathOperator {\Cay} {{\rm Cay}}

\newtheorem{theorem}{Theorem}
\newtheorem{lemma}[theorem]{Lemma}

\newtheorem{problem}[theorem]{Problem}

\newcommand{\diam}{{\rm diam}}

\date{}

\begin{document}

\title{On $\ell$-distance-balancedness of cubic Cayley graphs of dihedral groups}

\maketitle

\vspace{-19mm}

\begin{center}
\author{Gang Ma, \and Jianfeng Wang\footnote{Corresponding author}, \and Guang Li}\\[4mm]

\footnotesize School of Mathematics and Statistics, Shandong University of Technology, Zibo, China\\[2mm]
{\tt math$\_$magang@163.com (G.~Ma)\\
jfwang@sdut.edu.cn (J.F.~Wang)\\
lig@sdut.edu.cn (G.~Li)
}\\
\medskip
\end{center}

\begin{center}
\author{Sandi Klav\v{z}ar}\\[4mm]

\footnotesize Faculty of Mathematics and Physics, University of Ljubljana, Ljubljana, Slovenia\\[2mm]
\footnotesize  Institute of Mathematics, Physics and Mechanics, Ljubljana, Slovenia \\[2mm]
\footnotesize Faculty of Natural Sciences and Mathematics, University of Maribor, Slovenia \\[2mm]
{\tt sandi.klavzar@fmf.uni-lj.si}\\
\medskip

\end{center}

\begin{abstract}
A connected graph $\Gamma$ of diameter ${\rm diam}(\Gamma) \ge \ell$  is $\ell$-distance-balanced if $|W_{xy}(\Gamma)|=|W_{yx}(\Gamma)|$ for every $x,y\in V(\Gamma)$ with $d_{\Gamma}(x,y)=\ell$, where $W_{xy}(\Gamma)$ is the set of vertices of $\Gamma$ that are closer to $x$ than to $y$.
$\Gamma$ is said to be highly distance-balanced if it is $\ell$-distance-balanced for every $\ell\in [\diam(\Gamma)]$. It is proved that every cubic Cayley graph whose generating set is one of $\{a,a^{n-1},ba^r\}$ and $\{a^k,a^{n-k},ba^t\}$ is highly distance-balanced. This partially solves a problem posed by Miklavi\v{c} and \v{S}parl.
\end{abstract}


\medskip\noindent
{\bf AMS Subj.\ Class:} 05C12, 05C25

\section{Introduction}
\label{S:intro}

If $\Gamma= (V(\Gamma), E(\Gamma))$ is a connected graph and $x, y\in V(\Gamma)$, then the {\it distance} $d_{\Gamma}(x, y)$ between $x$ and $y$ is the number of edges on a shortest $(x,y)$-path. The diameter $\diam(\Gamma)$ of $\Gamma$ is the maximum distance between its vertices. The set $W_{xy}(\Gamma)$ contains the vertices that are closer to $x$ than to $y$, that is,
$W_{xy}(\Gamma)=\{w\in V(\Gamma):\ d_{\Gamma}(w,x) < d_{\Gamma}(w,y)\}$.
Vertices $x$ and $y$ are {\em balanced} if $|W_{xy}(\Gamma)| = |W_{yx}(\Gamma)|$.  For an integer $\ell \in [\diam(\Gamma)] = \{1,2,\ldots, \diam(\Gamma)\}$ we say that $\Gamma$ is $\ell$-{\em distance-balanced} if each pair of vertices $x,y\in V(\Gamma)$ with $d_{\Gamma}(x,y) = \ell$ is balanced. $\Gamma$ is said to be {\em highly distance-balanced} if it is $\ell$-distance-balanced for every $\ell\in [\diam(\Gamma)]$. $1$-distance-balanced graphs are simply called {\em distance-balanced} graphs.

Distance-balanced graphs were first considered by Handa~\cite{Handa:1999} back in 1999, while the term ``distance-balanced'' was proposed a decade later by Jerebic et al.\ in~\cite{Jerebic:2008}. The latter paper was the trigger for intensive research of distance-balanced graphs, see~\cite{Abiad:2016, Balakrishnan:2014, Balakrishnan:2009, Cabello:2011, cavaleri-2020, fernardes-2022, Ilic:2010, Kutnar:2006, Kutnar:2009, Kutnar:2014, Miklavic:2012, YangR:2009}. Moreover, distance-balanced graphs have motivated the introduction of the hitherto much-researched Mostar index~\cite{ali-2021, doslic-2018} and distance-unbalancedness of graphs~\cite{kr-2021, miklavic-2021, xu-2022}. In this context, distance-balanced graphs are the graphs with the Mostar index equal to 0. Distance-balanced graphs also coincide with  ``transmission regular graphs,'' see the survey~\cite{aouchiche-2024} on the latter class of graphs.

In~\cite{Frelih:2014}, Frelih generalized distance-balanced graphs to $\ell$-distance-balanced graphs.
Since then many researchers studied $\ell$-distance-balancedness of graphs from different aspects,
see~\cite{Frelih:2018,Jerebic:2021,MaG:2024,MaG:2024+,Miklavic:2018}. We emphasize that in~\cite{Miklavic:2018} some general results on $\ell$-distance balanced graphs are obtained and $\ell$-distance-balancedness of cubic graphs and graphs of diameter at most $3$ are studied.

The main object of our interest in this paper is Cayley graphs, so let's recapitulate the definitions. Let $G$ be a finite group and let $S\subseteq G$ be a {\em generating subset} with $S=S^{-1}$ and not containing the identity. Then the {\em Cayley graph} $\Cay(G;S)$ has the vertex set $G$, and $g\in G$ is adjacent to $h\in G$ whenever $g^{-1}h\in S$.

Kutnar et al.~\cite{Kutnar:2006} proved that every vertex-transitive graph is strongly distance-balanced. We do not give the definition of strongly distance-balancedness here, it suffices to state that, as the name suggests, every strongly distance-balanced graphs is distance-balanced. Since Cayley graphs are vertex-transitive, every Cayley graphs is hence distance-balanced. Moreover, Miklavi\v{c} and \v{S}parl~\cite{Miklavic:2018} proved that every Cayley graph of an abelian group is highly distance-balanced.

Given the above results, the question naturally arises as to what the situation is with distance-balancedness of Cayley graphs of non-abelian groups. In particular, Miklavi\v{c} and \v{S}parl posed the problem below, for which we recall that $D_{n}=C_n\rtimes C_2=\langle a,b\mid a^n=b^2=1, bab=a^{-1}\rangle$ is the {\em dihedral group} of order $2n$,
where $C_n=\langle a\rangle$ is a normal cyclic subgroup of $D_{n}$ of order $n$. Note that $D_{n}=\{1,a,\ldots,a^{n-1},b,ba,\ldots,ba^{n-1}\}$.

\begin{problem} {\rm \cite[Problem 6.7]{Miklavic:2018}}
\label{P:Ca-dihedral}
For each Cayley graph $\mathit\Gamma$ of a dihedral group determine all $\ell\ge 1$ such that $\mathit\Gamma$ is $\ell$-distance-balanced.
If this is too difficult in general, consider this problem for cubic Cayley graphs of dihedral groups.
\end{problem}

Our two main results which partially answer Problem~\ref{P:Ca-dihedral} read as follows.

\begin{theorem}\label{T:Cayley-cubic1}
If $S_1=\{a,a^{n-1},ba^r\}$, where $0\le r\le n-1$, then $\Cay(D_{n};S_1)$
is highly distance-balanced.
\end{theorem}

\begin{theorem}\label{T:Cayley-cubic2}
If $S_2=\{a^k,a^{n-k},ba^t\}$, where $1\le k<n/2$, $(k,n)=1$, and $0\le t\le n-1$, then $\Cay(D_{n};S_2)$ is highly distance-balanced.
\end{theorem}

The Cayley graphs $\Cay(D_{6};S_1)$ where $S_1= \{a, a^{5}, ba^2\}$, and $\Cay(D_{7};S_2)$ where $S_2=\{a^3, a^{4}, ba^5\}$, are shown in Fig.~\ref{F:Cay-n6-152}.
Theorems~\ref{T:Cayley-cubic1} and~\ref{T:Cayley-cubic2} are proved in Section~\ref{S:Cayley-cubic-proof}, while in the concluding section we list some open problems.

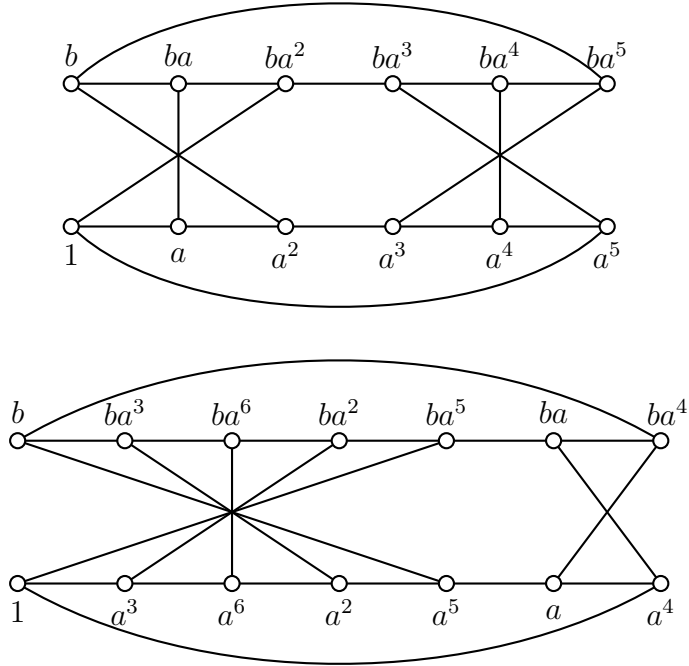
\begin{figure}[ht!]
\begin{center}
\begin{tikzpicture}[scale=0.95,style=thick,x=1cm,y=1cm]
\def\vr{3pt}
\begin{scope}[xshift=0cm, yshift=0cm] 
\coordinate(x1) at (0,0);
\coordinate(x2) at (1.5,0);
\coordinate(x3) at (3,0);
\coordinate(x4) at (4.5,0);
\coordinate(x5) at (6,0);
\coordinate(x6) at (7.5,0);
\coordinate(y1) at (0,2);
\coordinate(y2) at (1.5,2);
\coordinate(y3) at (3,2);
\coordinate(y4) at (4.5,2);
\coordinate(y5) at (6,2);
\coordinate(y6) at (7.5,2);
\draw (x1) -- (x6);
\draw (y1) -- (y6);
\draw (x1) -- (y3);
\draw (x2) -- (y2);
\draw (x3) -- (y1);
\draw (x4) -- (y6);
\draw (x5) -- (y5);
\draw (x6) -- (y4);
\draw (x1) .. controls (1.5,-1.5) and (6.0,-1.5) .. (x6);
\draw (y1) .. controls (1.5,3.5) and (6.0,3.5) .. (y6);
\foreach \i in {1,2,3,4,5,6}
{
\draw(x\i)[fill=white] circle(\vr);
\draw(y\i)[fill=white] circle(\vr);
}
\node at (0,-0.4) {$1$};
\node at (1.5,-0.4) {$a$};
\node at (3.0,-0.4) {$a^2$};
\node at (4.5,-0.4) {$a^3$};
\node at (6.0,-0.4) {$a^4$};
\node at (7.5,-0.4) {$a^5$};
\node at (0,2.4) {$b$};
\node at (1.5,2.4) {$ba$};
\node at (3.0,2.4) {$ba^2$};
\node at (4.5,2.4) {$ba^3$};
\node at (6.0,2.4) {$ba^4$};
\node at (7.5,2.4) {$ba^5$};
\end{scope}

\begin{scope}[xshift=-0.75cm, yshift=-5cm] 
\coordinate(x1) at (0,0);
\coordinate(x2) at (1.5,0);
\coordinate(x3) at (3,0);
\coordinate(x4) at (4.5,0);
\coordinate(x5) at (6,0);
\coordinate(x6) at (7.5,0);
\coordinate(x7) at (9,0);
\coordinate(y1) at (0,2);
\coordinate(y2) at (1.5,2);
\coordinate(y3) at (3,2);
\coordinate(y4) at (4.5,2);
\coordinate(y5) at (6,2);
\coordinate(y6) at (7.5,2);
\coordinate(y7) at (9,2);
\draw (x1) -- (x7);
\draw (y1) -- (y7);
\draw (x1) -- (y5);
\draw (x2) -- (y4);
\draw (x3) -- (y3);
\draw (x4) -- (y2);
\draw (x5) -- (y1);
\draw (x6) -- (y7);
\draw (x7) -- (y6);
\draw (x1) .. controls (2.5,-1.5) and (6.5,-1.5) .. (x7);
\draw (y1) .. controls (2.5,3.5) and (6.5,3.5) .. (y7);
\foreach \i in {1,2,3,4,5,6,7}
{
\draw(x\i)[fill=white] circle(\vr);
\draw(y\i)[fill=white] circle(\vr);
}
\node at (0,-0.4) {$1$};
\node at (1.5,-0.4) {$a^3$};
\node at (3.0,-0.4) {$a^6$};
\node at (4.5,-0.4) {$a^2$};
\node at (6.0,-0.4) {$a^5$};
\node at (7.5,-0.4) {$a$};
\node at (9.0,-0.4) {$a^4$};
\node at (0,2.4) {$b$};
\node at (1.5,2.4) {$ba^3$};
\node at (3.0,2.4) {$ba^6$};
\node at (4.5,2.4) {$ba^2$};
\node at (6.0,2.4) {$ba^5$};
\node at (7.5,2.4) {$ba$};
\node at (9.1,2.4) {$ba^4$};
\end{scope}

\end{tikzpicture}
\caption{$\Cay(D_{6};S_1)$ where $S_1=\{a,a^{5},ba^2\}$ (above); $\Cay(D_{7};S_2)$ where $S_2=\{a^3,a^{4},ba^5\}$ (below).}
\label{F:Cay-n6-152}
\end{center}
\end{figure}

\section{Proof of Theorems~\ref{T:Cayley-cubic1} and~\ref{T:Cayley-cubic2}}
\label{S:Cayley-cubic-proof}

For a non-negative integer $k$, we will use the notations $[k] = \{1,\ldots, k\}$ and $[k]_0 = \{0,\ldots, k-1\}$. Before proving Theorem~\ref{T:Cayley-cubic1}, there are three technical lemmas.

\begin{lemma}\label{L:dihe-1}
In $D_{n}$, if $i\in [n]_0$, then $(a^i)^{-1}=a^{n-i}$, $(ba^i)^{-1}=ba^i$,
and $ba^ib=a^{-i}$.
\end{lemma}

\begin{proof}
Because $a^n=b^2=1$, we have $(a^i)^{-1}=a^{-i}=a^{n-i}$ and $b^{-1}=b$.

Since $bab=a^{-1}$, we have $aba=b$ and $baba=1$. Hence $(ba^i)(ba^i)=ba^{i-1}abaa^{i-1}=ba^{i-1}ba^{i-1}=\cdots=baba=1$. So $(ba^i)^{-1}=ba^i$.

And since $(ba^i)^{-1}=ba^i$, we get $ba^iba^i=1$ and $ba^ib=a^{-i}$.
\end{proof}

\begin{lemma}\label{L:Cay-cubic1}
If $S_1=\{a,a^{n-1},ba^r\}$, $r\in [n]_0$, and if $i\in [n]_0$, then in $\Cay(D_{n};S_1)$,
\begin{enumerate}
\item[(1)] $a^i$ is adjacent to $a^{i-1}$, $a^{i+1}$, and $ba^{r-i}$; and
\item[(2)] $ba^i$ is adjacent to $ba^{i-1}$, $ba^{i+1}$, and $a^{r-i}$.
\end{enumerate}
\end{lemma}

\begin{proof}
Because (i) $(a^{i-1})^{-1}a^i=a\in S_1$, (ii) $(a^{i+1})^{-1}a^i=a^{-1}=a^{n-1}\in S_1$, and (iii) $(ba^{r-i})^{-1}a^i=ba^{r-i}a^i=ba^r\in S_1$, we get that
$a^i$ is adjacent to (i) $a^{i-1}$, (ii) $a^{i+1}$, and (iii) $ba^{r-i}$.

By Lemma~\ref{L:dihe-1} we have $ba^ib=a^{-i}$. Then because (i)
$(ba^{i-1})^{-1}(ba^i)=ba^{i-1}ba^i=a^{-(i-1)}a^i=a\in S_1$, (ii) $(ba^{i+1})^{-1}(ba^i)=ba^{i+1}ba^i=a^{-(i+1)}a^i=a^{-1}=a^{n-1}\in S_1$,
and (iii) $(a^{r-i})^{-1}(ba^i)=(ba^{r-i}b)(ba^i)=ba^r\in S_1$, we can conclude that
$ba^i$ is adjacent to (i) $ba^{i-1}$, (ii) $ba^{i+1}$, and (iii) $a^{r-i}$.
\end{proof}

\begin{lemma}\label{L:Cay-cubic2}
If $S_2 = \{a^k,a^{n-k},b(a^k)^r\}$, $1\le k<n/2$, $(k,n)=1$, $r\in [n]_0$, and if $i\in [n]_0$, then in $\Cay(D_{n};S_2)$,
\begin{enumerate}
\item[(1)] $a^{ik}$ is adjacent to $a^{(i-1)k}$, $a^{(i+1)k}$, and $ba^{(r-i)k}$; and
\item[(2)] $ba^{ik}$ is adjacent to $ba^{(i-1)k}$, $ba^{(i+1)k}$, and $a^{(r-i)k}$.
\end{enumerate}
\end{lemma}

\begin{proof}
Because
(i) $(a^{(i-1)k})^{-1}a^{ik}=a^k\in S_2$,
(ii) $(a^{(i+1)k})^{-1}a^{ik}=a^{-k}=a^{n-k}\in S_2$, and
(iii) $(ba^{(r-i)k})^{-1}a^{ik}=(ba^{(r-i)k})a^{ik}=ba^{rk}\in S_2$,
we get that
$a^{ik}$ is adjacent to (i) $a^{(i-1)k}$, (ii) $a^{(i+1)k}$, and (iii) $ba^{(r-i)k}$.

Since $ba^ib=a^{-i}$ (Lemma~\ref{L:dihe-1}), the computations
(i) $(ba^{(i-1)k})^{-1}(ba^{ik})=ba^{(i-1)k}ba^{ik}=a^{-(i-1)k}a^{ik}=a^k\in S_2$,
(ii) $(ba^{(i+1)k})^{-1}(ba^{ik})=ba^{(i+1)k}ba^{ik}=a^{-(i+1)k}a^{ik}=a^{-k}=a^{n-k}\in S_2$, and
(iii) $(a^{(r-i)k})^{-1}(ba^{ik})=(ba^{(r-i)k}b)(ba^{ik})=ba^{rk}\in S_2$ imply that
$ba^{ik}$ is adjacent to (i) $ba^{(i-1)k}$, (ii) $ba^{(i+1)k}$, and (iii) $a^{(r-i)k}$.
\end{proof}

Now it is time to prove Theorem~\ref{T:Cayley-cubic1}.

\begin{proof}[Proof of Theorem~\ref{T:Cayley-cubic1}]
Set $\Gamma=\Cay(D_{n};S_1)$. Because Cayley graphs are vertex-transitive, it suffices to prove that $|W_{1a^s}(\Gamma)| = |W_{a^s1}(\Gamma)|$ for $s\in [n-1]$ and
$|W_{1(ba^s)}(\Gamma)| = |W_{(ba^s)1}(\Gamma)|$ for $s\in [n]_0$.

Let $V_1=\{1,a,\ldots,a^{n-1}\}$ and $V_2=\{b,ba,\ldots,ba^{n-1}\}$.
Let $\Gamma_1=\Gamma[V_1]$ and $\Gamma_2=\Gamma[V_2]$ be the subgraphs of $\Gamma$ respectively induced by $V_1$ and $V_2$. Then $\Gamma_1$ and $\Gamma_2$ are isomorphic to cycles. In the light of the foregoing, the proof will be complete after proving the following two claims.

\medskip\noindent
{\bf Claim 1}: $|W_{1a^s}(\Gamma)| = |W_{a^s1}(\Gamma)|$, $s\in [n-1]$. \\
Let $s\in [n-1]$ and note that $1(ba^r)\in E(\Gamma)$ and $(a^s)(ba^{r-s})\in E(\Gamma)$. Set
\begin{align*}
W_{1a^s}(\Gamma_1) & = \{a^i\in V_1:\ d_{\Gamma_1}(1,a^i)<d_{\Gamma_1}(a^s,a^i)\}\,, \\
W_{a^s1}(\Gamma_1) & = \{a^i\in V_1:\ d_{\Gamma_1}(1,a^i)>d_{\Gamma_1}(a^s,a^i)\}\,, \\
W_{(ba^r)(ba^{r-s})}(\Gamma_2) & = \{ba^i\in V_2:\ d_{\Gamma_2}(ba^r,ba^i)<d_{\Gamma_2}(ba^{r-s},ba^i)\}\,, \\
W_{(ba^{r-s})(ba^r)}(\Gamma_2) & = \{ba^i\in V_2:\ d_{\Gamma_2}(ba^r,ba^i)>d_{\Gamma_2}(ba^{r-s},ba^i)\}\,.
\end{align*}
Since $\Gamma_1$ and $\Gamma_2$ are cycles, then $|W_{1a^s}(\Gamma_1)| = |W_{a^s1}(\Gamma_1)|$ and $|W_{(ba^r)(ba^{r-s})}(\Gamma_2)| = |W_{(ba^{r-s})(ba^r)}(\Gamma_2)|$. Moreover, the structure of $\Gamma$ yields
\begin{align*}
W_{1a^s}(\Gamma) & = W_{1a^s}(\Gamma_1)\cup W_{(ba^r)(ba^{r-s})}(\Gamma_2)\,,\\
W_{a^s1}(\Gamma) & = W_{a^s1}(\Gamma_1)\cup W_{(ba^{r-s})(ba^r)}(\Gamma_2)\,.
\end{align*}
We can conclude that $|W_{1a^s}(\Gamma)| = |W_{a^s1}(\Gamma)|$ for $s\in [n-1]$.

\medskip\noindent
{\bf Claim 2}: $|W_{1(ba^s)}(\Gamma)| = |W_{(ba^s)1}(\Gamma)|$, $s\in [n]_0$. \\
The proof of Claim 2 is divided into four cases according to the value of $s$.

\medskip\noindent
{\bf Case 1}: $s=r$.\\
In this case, $W_{1(ba^r)}(\Gamma) = V_1$ and $W_{(ba^r)1}(\Gamma)=V_2$, therefore $|W_{1(ba^r)}(\Gamma)| = |W_{(ba^r)1}(\Gamma)|$ as required.

\medskip\noindent
{\bf Case 2}: $s=r-1$. \\
Assume first that $n$ is even. then $d_{\Gamma}(1,a^i)=d_{\Gamma}(ba^{r-1},a^i)$ when $1\le i\le n/2$,
and $d_{\Gamma}(1,ba^i)=d_{\Gamma}(ba^{r-1},ba^i)$ when $r\le i\le r+n/2-1$. Consequently,
\begin{align*}
W_{1(ba^{r-1})}(\Gamma) & = V_1-\{a,a^2,\ldots,a^{n/2}\}\,,\\
W_{(ba^{r-1})1}(\Gamma) & = V_2-\{ba^r,ba^{r+1},\ldots,ba^{r+n/2-1}\}\,.
\end{align*}
We can conclude that $|W_{1(ba^{r-1})}(\Gamma)|=|W_{(ba^{r-1})1}(\Gamma)|$.

Assume second that $n$ is odd. Then $d_{\Gamma}(1,a^i)=d_{\Gamma}(ba^{r-1},a^i)$ when $1\le i\le (n-1)/2$,
and $d_{\Gamma}(1,ba^i)=d_{\Gamma}(ba^{r-1},ba^i)$ when $r\le i\le r+(n-1)/2-1$. Hence
\begin{align*}
W_{1(ba^{r-1})}(\Gamma) & = V_1-\{a,a^2,\ldots,a^{(n-1)/2}\}\,,\\
W_{(ba^{r-1})1}(\Gamma) & = V_2-\{ba^r,ba^{r+1},\ldots,ba^{r+(n-1)/2-1}\}\,,
\end{align*}
and we have the required conclusion $|W_{1(ba^{r-1})}(\Gamma)|=|W_{(ba^{r-1})1}(\Gamma)|$.

\medskip\noindent
{\bf Case 3}: $s=r+1$. \\
Assume first that $n$ is even. Then $d_{\Gamma}(1,a^{n-i})=d_{\Gamma}(ba^{r+1},a^{n-i})$ when $1\le i\le n/2$, and $d_{\Gamma}(1,ba^{r-i})=d_{\Gamma}(ba^{r+1},ba^{r-i})$ when $0\le i\le n/2-1$. Hence,
\begin{align*}
W_{1(ba^{r+1})}(\Gamma) & = V_1-\{a^{n-1},a^{n-2},\ldots,a^{n/2}\}\,,\\
W_{(ba^{r+1})1}(\Gamma) & = V_2-\{ba^r,ba^{r-1},\ldots,ba^{r-n/2+1}\}\,,
\end{align*}
and thus $|W_{1(ba^{r+1})}(\Gamma)|=|W_{(ba^{r+1})1}(\Gamma)|$.

Assume second that $n$ is odd. Now we get $d_{\Gamma}(1,a^{n-i})=d_{\Gamma}(ba^{r+1},a^{n-i})$ when $1\le i\le (n-1)/2$, and $d_{\Gamma}(1,ba^{r-i})=d_{\Gamma}(ba^{r+1},ba^{r-i})$ when $0\le i\le (n-1)/2-1$. Consequently,
\begin{align*}
W_{1(ba^{r+1})}(\Gamma) & = V_1-\{a^{n-1},a^{n-2},\ldots,a^{(n+1)/2}\}\,,\\
W_{(ba^{r+1})1}(\Gamma) & = V_2-\{ba^r,ba^{r-1},\ldots,ba^{r-(n-1)/2+1}\}\,,
\end{align*}
which yields the required conclusion $|W_{1(ba^{r+1})}(\Gamma)|=|W_{(ba^{r+1})1}(\Gamma)|$.

\medskip\noindent
{\bf Case 4.} $s\not\in\{r,r-1,r+1\}$.\\
Our aim is to prove the following two claims.

\medskip\noindent
{\bf Claim A}: If $i\in [n]_0$, then $a^i\in W_{1(ba^s)}(\Gamma)$ if and only if $ba^{s+i}\in W_{(ba^s)1}(\Gamma)$.

\medskip\noindent
{\bf Claim B}: If $i\in [n]_0$, then $a^i\in W_{(ba^s)1}(\Gamma)$ if and only if $ba^{s+i}\in W_{1(ba^s)}(\Gamma)$.

\medskip
Proving Claims A and B, we will get a bijection between $W_{1(ba^s)}(\Gamma)$ and  $W_{(ba^s)1}(\Gamma)$, which will in turn yield the desired conclusion $|W_{1(ba^{s})}| = |W_{(ba^{s})1}|$.

Note that
\begin{align*}
W_{1(ba^s)}(\Gamma) & = \{a^i\mid d_{\Gamma_1}(1,a^i)<d_{\Gamma_1}(a^{r-s},a^i)+1\}\ \cup \\
&\quad\ \{ba^i\mid d_{\Gamma_2}(ba^r,ba^i)+1<d_{\Gamma_2}(ba^s,ba^i)\},\\
W_{(ba^s)1}(\Gamma) & = \{a^i\mid d_{\Gamma_1}(1,a^i)>d_{\Gamma_1}(a^{r-s},a^i)+1\}\ \cup\\
&\quad\ \{ba^i\mid d_{\Gamma_2}(ba^r,ba^i)+1>d_{\Gamma_2}(ba^s,ba^i)\}.
\end{align*}
The proof of Claims A and B is divided into the following four cases according to the value of $s$ and $i$.

\medskip\noindent
{\bf Case 4.1}: $s\in [r-1]_0$, $i\in [r-s+1]_0$.\\
If $a^i\in W_{1(ba^s)}(\Gamma)$, then $d_{\Gamma}(1,a^i)<d_{\Gamma}(ba^s,a^i)$.
We prove that $ba^{s+i}\in W_{(ba^s)1}(\Gamma)$.
The path $1,a,a^2,\ldots, a^i$ is a shortest $(1,a^i)$-path, while
$$ba^s, a^{r-s}, a^{r-s-1},\ldots, a^i\quad \mbox{or}\quad ba^s, a^{r-s},a^{r-s+1},\ldots, a^i$$
is a shortest $(ba^s,a^i)$-path. Here and later ``or'' refers that one of the two possibilities holds according to the value of $r$ and $s$.

Further, the path $ba^s, ba^{s+1}, \ldots, ba^{s+i}$ is a shortest $(ba^s,ba^{s+i})$-path, while the path
$$1, ba^r, ba^{r-1},\ldots, ba^{s+i}\quad {\rm or}\quad 1, ba^r, ba^{r+1},\ldots, ba^{s+i}$$ is a shortest $(1,ba^{s+i})$-path. It follows that $$d_{\Gamma}(ba^s,ba^{s+i})=d_{\Gamma}(1,a^i) \quad \mbox{and} \quad d_{\Gamma}(1,ba^{s+i})=d_{\Gamma}(ba^s,a^i).$$
So $d_{\Gamma}(ba^s,ba^{s+i})<d_{\Gamma}(1,ba^{s+i})$ and $ba^{s+i}\in W_{(ba^s)1}(\Gamma)$.

The above discussion also implies that if $ba^{s+i}\in W_{(ba^s)1}(\Gamma)$, then $a^i\in W_{1(ba^s)}(\Gamma)$.
That is,  $a^i\in W_{1(ba^s)}(\Gamma)$ if and only if $ba^{s+i}\in W_{(ba^s)1}(\Gamma)$, which demonstrates Claim A in this case.

If $a^i\in W_{(ba^s)1}(\Gamma)$, then $d_{\Gamma}(ba^s,a^i)<d_{\Gamma}(1,a^i)$.
We prove that $ba^{s+i}\in W_{1(ba^s)}(\Gamma)$. The path
$ba^s, a^{r-s},a^{r-s-1},\ldots, a^i$ is a shortest $(ba^s,a^i)$-path, and
$$1, a, a^2, \ldots, a^i\quad {\rm or}\quad 1,a^{n-1}, a^{n-2}, \ldots, a^i$$
is a shortest $(1,a^i)$-path. Further, $1, ba^r, ba^{r-1},\ldots, (ba^{s+i})$
is a shortest $(1,ba^{s+i})$-path, and
$$ba^s, ba^{s+1},\ldots, ba^{s+i} \quad {\rm or}\quad ba^s, ba^{s-1},\ldots,  ba^{s+i}$$
is a shortest $(ba^s,ba^{s+i})$-path. Hence, $$d_{\Gamma}(1,ba^{s+i}) = d_{\Gamma}(ba^s,a^i) \quad \mbox{and} \quad d_{\Gamma}(ba^s,ba^{s+i}) = d_{\Gamma}(1,a^i).$$
So $d_{\Gamma}(1,ba^{s+i})<d_{\Gamma}(ba^s,ba^{s+i})$ and $ba^{s+i}\in W_{1(ba^s)}(\Gamma)$.

The above discussion also yields that if $ba^{s+i}\in W_{1(ba^s)}(\Gamma)$, then $a^i\in W_{(ba^s)1}(\Gamma)$.
That is, $a^i\in W_{(ba^s)1}(\Gamma)$ if and only if $ba^{s+i}\in W_{1(ba^s)}(\Gamma)$. This proves Claim B in this case.

\medskip\noindent
{\bf Case 4.2}: $s\in [r-1]_0$, $r-s<i\le n-1$.\\
If $a^i\in W_{1(ba^s)}(\Gamma)$, then $d_{\Gamma}(1,a^i)<d_{\Gamma}(ba^s,a^i)$.
We prove that $ba^{s+i}\in W_{(ba^s)1}(\Gamma)$.
The path $1, a^{n-1}, \ldots, a^i$ is a shortest $(1,a^i)$-path, and
$$ba^s, a^{r-s}, a^{r-s+1}, \ldots, a^i \quad {\rm or}\quad ba^s, a^{r-s}, a^{r-s-1}, \ldots, a^i$$
is a shortest $(ba^s,a^i)$-path. In addition, the path $ba^s, ba^{s-1}, \ldots, ba^{s+i-n}$ is a shortest $(ba^s,ba^{s+i})$-path while
$$1, ba^r, ba^{r-1}, \ldots, ba^{s+i} \quad {\rm or}\quad 1, ba^r, ba^{r+1}, \ldots,  ba^{s+i}$$
is a shortest $(1,ba^{s+i})$-path. Consequently,
$$d_{\Gamma}(ba^s,ba^{s+i})=d_{\Gamma}(1,a^i) \quad \mbox{and} \quad
d_{\Gamma}(1,ba^{s+i})=d_{\Gamma}(ba^s,a^i).$$
So $d_{\Gamma}(ba^s,ba^{s+i})<d_{\Gamma}(1,ba^{s+i})$ and $ba^{s+i}\in W_{(ba^s)1}(\Gamma)$.

Again we also see that if $ba^{s+i}\in W_{(ba^s)1}(\Gamma)$, then $a^i\in W_{1(ba^s)}(\Gamma)$ and we can conclude that $a^i\in W_{1(ba^s)}(\Gamma)$ if and only if $ba^{s+i}\in W_{(ba^s)1}(\Gamma)$. This establishes Claim A in this case.

If $a^i\in W_{(ba^s)1}(\Gamma)$, then $d_{\Gamma}(ba^s,a^i)<d_{\Gamma}(1,a^i)$.
We prove that $ba^{s+i}\in W_{1(ba^s)}(\Gamma)$. The path
$ba^s, a^{r-s}, a^{r-s+1}, \ldots, a^i$
is a shortest $(ba^s,a^i)$-path, and
$$1, a, a^2, \ldots, a^i \quad {\rm or}\quad  1,a^{n-1}, a^{n-2}, \ldots, a^i$$
is a shortest $1,a^i$-path. Further, $1, ba^r, ba^{r+1}, \ldots, ba^{s+i}$
is a shortest $(1,ba^{s+i})$-path, and
$$ba^s, ba^{s+1}, \ldots, ba^{s+i} \quad {\rm or}\quad ba^s, ba^{s-1}, \ldots, ba^{s+i}$$
is a shortest $(ba^s,ba^{s+i})$-path. This means that $$d_{\Gamma}(1,ba^{s+i})=d_{\Gamma}(ba^s,a^i) \quad \mbox{and} \quad d_{\Gamma}(ba^s,ba^{s+i})=d_{\Gamma}(1,a^i).$$ Hence, $d_{\Gamma}(1,ba^{s+i})<d_{\Gamma}(ba^s,ba^{s+i})$ and $ba^{s+i}\in W_{1(ba^s)}(\Gamma)$.

Using the above discussion we also infer that if $ba^{s+i}\in W_{1(ba^s)}(\Gamma)$, then $a^i\in W_{(ba^s)1}(\Gamma)$. So $a^i\in W_{(ba^s)1}(\Gamma)$ if and only if $ba^{s+i}\in W_{1(ba^s)}(\Gamma)$ and Claim B is verified in this case.

\medskip\noindent
{\bf Case 4.3}: $r+2\le s\le n-1$, $i\in [r-s+n+1]_0$.\\
If $a^i\in W_{1(ba^s)}(\Gamma)$, then $d_{\Gamma}(1,a^i)<d_{\Gamma}(ba^s,a^i)$.
We prove that $ba^{s+i}\in W_{(ba^s)1}(\Gamma)$. The path
$1, a, a^2, \ldots, a^i$ is a shortest $(1,a^i)$-path, and
$$ba^s, a^{r-s}, a^{r-s-1},\ldots, a^i \quad {\rm or}\quad ba^s, a^{r-s}, a^{r-s+1}, \ldots, a^i$$
is a shortest $(ba^s,a^i)$-path. Furthermore, the path
$ba^s, ba^{s+1}, \ldots, ba^{s+i}$ is a shortest $(ba^s,ba^{s+i})$-path, and
$$1, ba^r, ba^{r-1}, \ldots, ba^{s+i} \quad {\rm or}\quad 1, ba^r, ba^{r+1}, \ldots, ba^{s+i}$$
is a shortest $(1,ba^{s+i})$-path. From this we deduce that $$d_{\Gamma}(ba^s,ba^{s+i})=d_{\Gamma}(1,a^i) \quad {\rm and}  \quad d_{\Gamma}(1,ba^{s+i})=d_{\Gamma}(ba^s,a^i).$$
So $d_{\Gamma}(ba^s,ba^{s+i})<d_{\Gamma}(1,ba^{s+i})$ and $ba^{s+i}\in W_{(ba^s)1}(\Gamma)$. We further get that if $ba^{s+i}\in W_{(ba^s)1}(\Gamma)$, then $a^i\in W_{1(ba^s)}(\Gamma)$. Hence $a^i\in W_{1(ba^s)}(\Gamma)$ if and only if $ba^{s+i}\in W_{(ba^s)1}(\Gamma)$ which establishes Claim A.

If $a^i\in W_{(ba^s)1}(\Gamma)$, then $d_{\Gamma}(ba^s,a^i)<d_{\Gamma}(1,a^i)$.
We prove that $ba^{s+i}\in W_{1(ba^s)}(\Gamma)$. The path
$ba^s, a^{r-s+n}, a^{r-s+n-1}, \ldots, a^i$
is a shortest $(ba^s,a^i)$-path, and
$$1, a, a^2, \ldots, a^i \quad {\rm or}\quad 1, a^{n-1}, a^{n-2},\ldots, a^i$$ is a shortest $(1,a^i)$-path. Moreover, the path
$1, ba^r, ba^{r-1}, \ldots, ba^{s+i}$
is a shortest $(1,ba^{s+i})$-path, and
$$ba^s, ba^{s+1}, \ldots, ba^{s+i} \quad {\rm or}\quad ba^s, ba^{s-1}, \ldots,  ba^{s+i}$$
is a shortest $(ba^s,ba^{s+i})$-path. Hence $$d_{\Gamma}(1,ba^{s+i})=d_{\Gamma}(ba^s,a^i) \quad {\rm and}  \quad d_{\Gamma}(ba^s,ba^{s+i})=d_{\Gamma}(1,a^i)$$ which in turn implies that $d_{\Gamma}(1,ba^{s+i})<d_{\Gamma}(ba^s,ba^{s+i})$ and $ba^{s+i}\in W_{1(ba^s)}(\Gamma)$. We also get that if $ba^{s+i}\in W_{1(ba^s)}(\Gamma)$ then $a^i\in W_{(ba^s)1}(\Gamma)$. That is, $a^i\in W_{(ba^s)1}(\Gamma)$ if and only if $ba^{s+i}\in W_{1(ba^s)}(\Gamma)$. Claim B follows in this case.

\medskip\noindent
{\bf Case 4.4}: $r+2\le s\le n-1$, $r-s+n<i\le n-1$.\\
If $a^i\in W_{1(ba^s)}(\Gamma)$, then $d_{\Gamma}(1,a^i)<d_{\Gamma}(ba^s,a^i)$.
We prove that $ba^{s+i}\in W_{(ba^s)1}(\Gamma)$. The path
$1, a^{n-1}, \ldots, a^i$ is a shortest $(1,a^i)$-path, and
$$ba^s, a^{r-s}, a^{r-s+1}, \ldots, a^i \quad {\rm or}\quad ba^s, a^{r-s}, a^{r-s-1}, \ldots, a^i$$
is a shortest $(ba^s,a^i)$-path. Next, $ba^s, ba^{s-1}, \ldots, ba^{s+i-n}$ is a shortest $(ba^s,ba^{s+i})$-path, and
$$1, ba^r, ba^{r-1}, \ldots, ba^{s+i} \quad {\rm or}\quad 1, ba^r, ba^{r+1}, \ldots,  ba^{s+i}$$
is a shortest $(1,ba^{s+i})$-path. Hence, $$d_{\Gamma}(ba^s,ba^{s+i})=d_{\Gamma}(1,a^i) \quad \mbox{and} \quad d_{\Gamma}(1,ba^{s+i})=d_{\Gamma}(ba^s,a^i).$$
So $d_{\Gamma}(ba^s,ba^{s+i})<d_{\Gamma}(1,ba^{s+i})$ and $ba^{s+i}\in W_{(ba^s)1}(\Gamma)$. We also get that if $ba^{s+i}\in W_{(ba^s)1}(\Gamma)$, then $a^i\in W_{1(ba^s)}(\Gamma)$.
That is to say, $a^i\in W_{1(ba^s)}(\Gamma)$ if and only if $ba^{s+i}\in W_{(ba^s)1}(\Gamma)$. Claim A follows.

If $a^i\in W_{(ba^s)1}(\Gamma)$, then $d_{\Gamma}(ba^s,a^i)<d_{\Gamma}(1,a^i)$.
We prove that $ba^{s+i}\in W_{1(ba^s)}(\Gamma)$.
The path $ba^s, a^{r-s}, a^{r-s+1}, \ldots, a^i$
is a shortest $(ba^s,a^i)$-path, and
$$1, a, a^2, \ldots, a^i \quad {\rm or}\quad 1, a^{n-1}, a^{n-2}, \ldots, a^i$$ is a shortest $(1,a^i)$-path. Next, $1, ba^r, ba^{r+1}, \ldots, ba^{s+i}$
is a shortest $(1,ba^{s+i})$-path, and
$$ba^s, ba^{s+1}, \ldots, ba^{s+i} \quad {\rm or}\quad ba^s, ba^{s-1}, \ldots, ba^{s+i}$$
is a shortest $(ba^s,ba^{s+i})$-path.
So, $$d_{\Gamma}(1,ba^{s+i})=d_{\Gamma}(ba^s,a^i) \quad  \mbox{and} \quad
d_{\Gamma}(ba^s,ba^{s+i})=d_{\Gamma}(1,a^i).$$
Hence, $d_{\Gamma}(1,ba^{s+i})<d_{\Gamma}(ba^s,ba^{s+i})$ and $ba^{s+i}\in W_{1(ba^s)}(\Gamma)$. Moreover, we also get that if $ba^{s+i}\in W_{1(ba^s)}(\Gamma)$, then $a^i\in W_{(ba^s)1}(\Gamma)$.
That is, $a^i\in W_{(ba^s)1}(\Gamma)$ if and only if $ba^{s+i}\in W_{1(ba^s)}(\Gamma)$. Claim B follows also in this case.
\end{proof}

It remains to prove Theorem~\ref{T:Cayley-cubic2}.

\begin{proof}[Proof of Theorem~\ref{T:Cayley-cubic2}]
Let $\Gamma_1=\Cay(D_{n};S_1)$, where $S_1=\{a,a^{n-1},ba^r\}$, and let $\Gamma_2=\Cay(D_{n};S_2)$, where $S_2=\{a^k,a^{n-k},ba^t\}$.
We will prove that $\Gamma_1$ is isomorphic to $\Gamma_2$.

Let $r$ be an integer such that $t=kr (\bmod) n$. Note that $r$ exists because we have assumed that $(k,n)=1$. Then $ba^t=ba^{kr}$.

Let $\theta:V(\Gamma_1)\to V(\Gamma_2)$ be a bijection defined by $\theta(a^i)=a^{ik}$ and $\theta(ba^i)=ba^{ik}$, $i\in [n]_0$.
Let  $\phi:E(\Gamma_1)\to E(\Gamma_2)$ be a bijection defined by $\phi(a^{i}a^{i+1}) = a^{ik}a^{(i+1)k}$,
$\phi((ba^{i})(ba^{i+1})) = (ba^{ik})(ba^{(i+1)k})$, and $\phi(a^{i}(ba^{r-i})) = a^{ik}(ba^{(r-i)k})$, $i\in [n]_0$.

For $i\in [n]_0$ we have
\begin{align*}
\phi(a^{i}a^{i+1}) & = a^{ik}a^{(i+1)k}=\theta(a^{i})\theta(a^{i+1})\,, \\
\phi((ba^{i})(ba^{i+1})) & = (ba^{ik})(ba^{(i+1)k})=\theta(ba^{i})\theta(ba^{i+1})\,, \\
\phi(a^{i}(ba^{r-i})) & = a^{ik}(ba^{(r-i)k})=\theta(a^{i})\theta(ba^{r-i})\,.
\end{align*}
This proves that $\Gamma_1$ and $\Gamma_2$ are isomorphic, hence Theorem~\ref{T:Cayley-cubic1} implies the result.
\end{proof}

\section{Concluding remarks}
\label{S:conluding}

Distance-balancedness of cubic Cayley graphs of dihedral groups remains to be considered for the other two types. More precisely:

\begin{problem}
Study the $\ell$-distance-balancedness of $\Cay(D_{n};\{a^{n/2},ba^{k_1},ba^{k_2}\})$ and of $\Cay(D_{n};\{ba^{k_1},ba^{k_2},ba^{k_3}\})$.
\end{problem}

Of course, we also have:

\begin{problem}
\label{prob:non-cubic}
Study the $\ell$-distance-balancedness of non cubic Cayley graphs of dihedral groups.
\end{problem}

With respect to Problem~\ref{prob:non-cubic} we point to the following example. The Cayley graph $\Cay(D_{9};\{a^3,a^6,b,ba^2,ba^3\})$ is of diameter $3$ and is neither $2$-distance-balanced nor $3$-distance-balanced, see~\cite[Fig.~2]{Miklavic:2018}.

More generally, we also pose:

\begin{problem}
Study the $\ell$-distance-balancedness of cubic Cayley graphs of groups except dihedral groups.
\end{problem}

With respect to the last problem, consider the next examples. The Cayley graph $\Cay(A_4;S)$, where $S=\{(1\ 2\ 3),(1\ 3\ 2),(1\ 2)(3\ 4)\}$, is a cubic graph of diameter $3$. Surprisingly, it is $1$-distance-balanced, $3$-distance-balanced, but it is not $2$-distance-balanced. As another example consider $\Cay(S_4;\{(1\ 2),(2\ 4),(1\ 2)(3\ 4)\})$, which is a cubic graph of diameter $4$. However, this Cayley graph is $1$-distance-balanced, $2$-distance-balanced, but it is neither $3$-distance-balanced nor $4$-distance-balanced, see~\cite[Fig.~1]{Miklavic:2018}.

\section*{Acknowledgments}

Gang Ma is supported by the Natural Science Foundation of Shandong Province (ZR2022MA077).
Jianfeng Wang is supported by the NSFC (12371353),  the Special Fund for the Taishan Scholars Project and the IC Program of Shandong Institutions of Higher Learning IC Program for Innovative Young Talents, and the Shandong Province Postgraduate Education Reform Project (SDYKC2023107) and Undergraduate Teaching Reform Research Project of Shandong (Z2022027).
Guang Li is supported by the NSFC (12301449) of China.
Sandi Klav\v{z}ar was supported by the Slovenian Research Agency (ARIS) under the grants P1-0297, N1-0355, and N1-0285.

%
%
%
%

\end{document}